\documentclass[11pt]{amsart}
\usepackage{amsmath,amssymb,amsthm}
\usepackage[mathscr]{eucal}
\usepackage{color,graphics,srcltx}
\usepackage{bm}
\makeatletter
\newtheorem{theorem}{Theorem}[section]
\newtheorem{lemma}[theorem]{Lemma}

\newtheorem{proposition}[theorem]{Proposition}
\newtheorem{remark}[theorem]{Remark}

\theoremstyle{definition}

\makeatother

\def\Id{I}

\def\ZZ{{\mathbb Z}}
\def\CC{{\mathbb C}}

\def\QQ{{\mathbb Q}}

\def\Lin{\mathop{\rm Lin}}
\def\Sp{\mathop{\rm Sp}}
\def\Tr{\mathop{\rm Tr}}

\def\D{\mathscr D}
\def\K{\mathscr K}

\begin{document}
\title[Spectrum nonincreasing maps]{Spectrum nonincreasing maps on matrices}

\author{Gregor Dolinar}
\address[Gregor Dolinar]{Faculty of Electrical Engineering, University of Ljubljana, Tr{\v{z}}a{\v{s}}ka cesta 25, SI-1000 Ljubljana, Slovenia}
\email{gregor.dolinar@fe.uni-lj.si}

\author{Jinchuan Hou}
\address[Jinchuan Hou]{School of
Mathematics, Taiyuan University of Technology,
 Taiyuan 030024, P. R. China} \email{jinchuanhou@yahoo.com.cn}

\author{Bojan Kuzma}
\address[Bojan Kuzma]{$^{1}$University of Primorska, FAMNIT, Glagolja\v ska 8, SI-6000 Koper, \and
$^{2}$Institute of Mathematics, Physics and Mechanics, Department of
Mathematics, Jadranska~19, SI-1000 Ljubljana, Slovenia.}
\email{bojan.kuzma@pef.upr.si}

\author{Xiaofei Qi}
\address[Xiaofei Qi]{School of Mathematical science,
 Shanxi University, Taiyuan 030006, P. R. China} \email{xiaofeiqisxu@yahoo.com.cn}

\thanks{This work was partially supported by a joint Slovene-Chinese Grant \mbox{BI-CN/11-13/9-2}, National Natural Science Foundation of China
(11171249, 11101250, 11201329) and a grant from International
Cooperation Program in Sciences and Technology  of Shanxi
(2011081039).}

\begin{abstract}Maps $\Phi$ which do not increase the spectrum on
complex matrices in a sense that
$\Sp(\Phi(A)-\Phi(B))\subseteq\Sp(A-B)$ are classified.

\end{abstract}

\maketitle

\section{Introduction}

Let $M_n(\CC)$ be the set of all $n\times n$ matrices over the
complex field $\CC$, and let $\Sp(X)$ be the spectrum of $X\in
M_n(\CC)$. In \cite{MM}, Marcus and Moyls proved that every
linear map $\Phi\colon M_n(\CC)\rightarrow M_n(\CC)$ preserving
eigenvalues (counting multiplicities) is either an isomorphism
or an anti-isomorphism. Furthermore, by using their result, one
can show that every linear map $\Phi\colon M_n(\CC)\rightarrow
M_n(\CC)$ preserving spectrum of matrices (that is,
$\Sp(\Phi(A))=\Sp(A)$ for all $A\in M_n(\CC)$)  also has the
standard form, that is, it is an isomorphism or an anti-isomorphism.

This result has been generalized in different directions.
Instead of matrix algebras, the algebras of all bounded linear
operators on a complex Banach space were considered, see for
example \cite{A,AM,JS,OS} and the references therein. Also,
instead of linear or additive preservers, general preservers
(without linearity and additivity assumption) of spectrum on
$M_n(\CC)$ were considered. Baribeau and
Ransford~\cite{baribeua-ransford} proved that a spectrum
preserving ${\mathcal C}^1$ diffeomorphism from an open subset
of $M_n(\CC)$ into $M_n(\CC)$ has the standard form. Mr\v cun
showed in \cite{M} that, if $\Phi\colon M_n(\CC)\rightarrow
M_n(\CC)$ is a Lipschitz map with $\Phi(0)=0$ such that
$\Sp(\Phi(A)-\Phi(B))\subseteq \Sp(A-B)$ for all $A,B\in
M_n(\CC)$ then $\Phi$ has the standard form. Costara in
\cite{costara} improved the above result by relaxing
Lipschitzian property to continuity. Recently, the continuity
of the map was replaced by surjectivity.  Namely, in
\cite{BDS}, Bendaoud, Douimi and Sarih proved that a surjective
map $\Phi\colon M_n(\CC)\rightarrow M_n(\CC)$ satisfying $\Phi(0)=0$
and $\Sp(\Phi(A)-\Phi(B))\subseteq \Sp(A-B)$ for all $A,B\in
M_n(\CC)$ has the standard  form. We should mention here that
the condition $\Phi(0)=0$ is a harmless normalization: If
$\Psi$ is any map with $\Sp(\Psi(A)-\Psi(B))\subseteq\Sp(A-B)$,
then $\Phi(X):=\Psi(X)-\Psi(0)$ also satisfies this property.

\if false Aupetit and Mouton \cite{AM} extended the result of
Jafarian and Sourour to primitive Banach algebras with minimal
ideals. Later, Aupetit \cite{A} proved that every
spectrum-preserving linear surjection between von Neumann algebras
is a Jordan isomorphism.\fi

It is our aim to prove the following generalization of
\cite[Theorem 1]{costara} and \cite[Theorem 1.3]{BDS}, in which
the maps considered are neither continuous nor surjective.

\begin{theorem}
 Let $n\ge 2$. Suppose that $\Phi:M_n(\CC)\to
M_n(\CC)$ is a map with $\Phi(0)=0$ and
$$\Sp(\Phi(A)-\Phi(B))\subseteq\Sp(A-B)\quad{\rm for\ \ all}\quad A,B\in M_n(\CC).$$
Then there exists an invertible matrix $S\in M_n(\CC)$ such that
either $\Phi(A)=SAS^{-1}$ for all $A\in M_n(\CC)$ or
$\Phi(A)=SA^tS^{-1}$  for all $A\in M_n(\CC)$, where $A^t$ denotes
the transpose of $A$.
\end{theorem}

\begin{remark} It was shown by Costara~\cite{costara} that the maps which satisfy
$\Sp(\Phi(A)-\Phi(B))\supseteq\Sp(A-B)$ are also linear and
 of a standard form.
\end{remark}

\section{Structural features of bases of matrix algebras}

In this section, some  features of bases of $M_n(\CC)$  will be
given, which are useful for proving our main result.

Recall that complex numbers $\alpha_1,\dots,\alpha_n$ are
linearly independent over $\ZZ$, the ring of integers, if the
only possibility that $\sum_{i=1}^n z_i \alpha_i=0$ for some
numbers $z_1,\dots,z_n\in\ZZ$ is $z_1=z_2=\dots=z_n=0$.

\begin{lemma}\label{lem:countable}
Let $\alpha_1(t),\dots,\alpha_n(t)$ be $n$  linearly
independent analytic functions in $t$, defined in a
neighborhood of a closed unit disc~$\Delta$. Then the set of
all parameters $t_0\in\Delta$ such that the complex numbers
$\alpha_1(t_0),\dots,\alpha_n(t_0)$ are linearly dependent over
$\ZZ$ is at most countable.
\end{lemma}

\begin{proof}
An analytic function either vanishes identically or it has only
finitely many zeros in  compact subsets. Since
$\alpha_1(t),\dots,\alpha_n(t)$ are linearly independent, none of
them can be identically zero, hence each of them has only finitely
many zeros in $\Delta$. Let $t_1,\dots,t_k$ be all the zeros of the
product $\alpha_1(t)\alpha_2(t)\cdots \alpha_n(t)$ and let
$t_0\in\Delta\setminus\{t_1,\dots,t_k\}$. It is straightforward that
there exist integers $z_1,\dots,z_n\in\ZZ$, with $z_n\ne0$, such
that
$$z_1\alpha_1(t_0)+\dots+z_n\alpha_n(t_0)=0$$
if and only if there exist rational numbers $p_i \in \QQ$ (in
fact, $p_i=\frac{z_i}{z_n}$) such that the meromorphic
functions
$\widehat{\alpha_i}(t):=\frac{\alpha_i(t)}{\alpha_n(t)}$
satisfy
$$p_1\widehat{\alpha_1}(t_0)+\dots+p_{n-1}\widehat{\alpha_{n-1}}(t_0)+1=0.$$
By linear independency, the meromorphic function
$\alpha_{p_1,\dots,p_{n-1}}\colon t\mapsto
p_1\widehat{\alpha_1}(t)+\dots+p_{n-1}\widehat{\alpha_{n-1}}(t)+1$
is nonzero. Thus it has at most finitely many zeros on compact
subsets of $\Delta\setminus\{t_1,\dots,t_k\}$, and so it has at
most countably many zeros in $\Delta$. Since there are
countably many functions $\alpha_{p_1,\dots,p_{n-1}}$ as
$p_1,\dots,p_{n-1}$ varies along rational numbers, we get at
most $|\QQ|\times |\QQ|=|\QQ|$, i.e., at most countably many
points  which annihilate one among the functions
$\alpha_{p_1,\dots,p_{n-1}}$. Adding also $\{t_1,\dots,t_k\}$,
we see that there are  at most countably many points
$t\in\Delta$, for which the complex numbers
$\alpha_1(t),\dots,\alpha_n(t)$ are linearly dependent with
$z_n\ne0$. If $z_n=0$ we are seeking for linear dependence of
scalars $\alpha_1(t),\dots,\alpha_{n-1}(t)$. By the same
argument as before, there are at most countably many such
$t$'s, under additional hypothesis that $z_{n-1}\ne0$.
Proceeding inductively backwards, there are at most countably
many parameters $t\in\Delta$ for which
$\alpha_1(t),\dots,\alpha_n(t)$ are linearly dependent.
\end{proof}

To formulate the next technical lemma, we introduce the
following notation: given the set $\Omega \subseteq \CC$ of
cardinality $n$, let $\vec{\Omega} \subseteq \CC^n$ be the set
of  $n!$  column vectors in $\CC^n$  such that the set of their
components, relative to a standard basis, equals $\Omega$.
For example, if $\Omega = \{1,2\}$, then
$\vec\Omega = \{\left(
                                                                                                                 \begin{smallmatrix}
                                                                                                                   1 \\
                                                                                                                   2 \\
                                                                                                                 \end{smallmatrix}
                                                                                                               \right), \left(
                                                                                                                          \begin{smallmatrix}
                                                                                                                            2 \\
                                                                                                                            1 \\
                                                                                                                          \end{smallmatrix}
                                                                                                                        \right)
                                                                                                                        \}$.

\begin{lemma}\label{lem:permutations}
There exist $n$ sets $\Omega_1, \dots, \Omega_n \subseteq \CC$, each of cardinality $n$, 
such that
 $n$ vectors $x_1, \dots, x_n$ are linearly independent for any choice of $x_i \in \vec \Omega_i$, $i=1, \dots, n$.
\end{lemma}
\begin{proof}
It suffices to prove that there exist $n$  column vectors
$x_1,\dots,x_n\in\CC^n$ such that (i) each of them has pairwise
distinct components in a standard basis, and  (ii) for
arbitrary permutation matrices $P_1,\dots,P_n$, the vectors
$P_1x_1,\dots,P_n x_n$ are linearly independent. Let $x_1$ be
an arbitrary column vector with pairwise distinct components.
Assume $x_1,\dots,x_k$, $k<n$, are column vectors with the
properties from the lemma. Then $\Lin\{P_1x_1,\dots,P_kx_k\}$
(the linear span of vectors $P_1x_1,\dots,P_kx_k$) with $P_1,
\dots, P_k$ permutation matrices is a $k$-dimensional linear
plane, and
there are at most finitely many 
such distinct $k$-dimensional subspaces of $\CC^n$ for all possible choices of permutation matrices. Since the union
of such $k$-dimensional subspaces is a closed proper subset of the whole space $\CC^n$, there exists a
vector $x_{k+1}\in\CC^n$ which does not belong to any
of these subspaces and has pairwise distinct components.
Let $P_1, \dots, P_k, P_{k+1}$ be arbitrary permutation
matrices. If $P_{k+1}x_{k+1}\in\Lin\{P_1x_1,\dots,P_kx_k\}$,
then $x_{k+1}\in P_{k+1}^{-1}\Lin\{P_1x_1,\dots,P_kx_k\}=
\{P_{k+1}^{-1}P_1x_1,\dots,P_{k+1}^{-1}P_kx_k\}$, and since
$P_{k+1}^{-1}P_i$, $i=1,\dots, k$, are again permutation
matrices, we obtain a contradiction. Thus $P_1x_1,\dots,P_{k+1}
x_{k+1}$ are linearly independent. By induction on $n$, the
lemma is true.
\end{proof}

Let ${\mathscr K}$ denotes the set of all matrices in
$M_n(\CC)$ which have $n$ pairwise distinct eigenvalues,
linearly independent over $\ZZ$. For example, since
$\pi=3^{\cdot}14\dots$ is not algebraic, the diagonal matrix
diag$(1,\pi,\pi^2,\dots,\pi^{n-1})\in{\mathscr K}$. We will use
the following properties of the set ${\mathscr K}$.

\begin{lemma}\label{lem:K'}
Let $B_1,\dots,B_{m}\in{\mathscr K}$  be a finite sequence of
matrices in $M_n(\CC)$. Then, the set of matrices $X\in {\mathscr
K}$ such that $B_k-X\in{\mathscr K}$ for each $k=1,\dots,m$ is dense
in $M_n(\CC)$.
\end{lemma}

\begin{proof}
Choose any $A\in M_n(\CC)$ and let $\varepsilon >0$. We will show
that there exists $X \in {\mathscr K}$ such that  $\|A-X\| <
\varepsilon$ and $B_k-X\in{\mathscr K}$ for each $k=1,\dots,m$.

Let $B_0 =0$. It is easy to see that the set ${\mathscr D}_n$
of matrices with $n$ distinct eigenvalues is dense and  open in
$M_n(\CC)$. Hence we may find $A_0\in {\mathscr D}_n$
arbitrarily close to $A$. Since ${\mathscr D}_n$ is open, each
neighborhood of $A_0$,  which is small enough,  contains only
matrices from ${\mathscr D}_n$. So we may find $A_1\in
{\mathscr D}_n$ arbitrarily close to $A_0$ such that
$B_1-A_1\in{\mathscr D}_n$.  Proceeding recursively, there
exists $\hat{A}\in M_n(\CC)$ such that
$\|A-\hat{A}\|<\varepsilon$ and that, moreover,
$\hat{A}-B_k\in{\mathscr D}_n$ for each
$k=0,\dots,m$.
Without loss of generality we write in the sequel $A$ instead
of $\hat{A}$, that is, we  assume that $A, A-B_k$ are all in
$\D_n$, $k=1, \dots, m$.

Note that, for each fixed $k$, $0\le k \le m$, by
Lemma~\ref{lem:permutations}, there exist $n$ matrices $C_{k,1},
\dots, C_{k,n}$ such that the spectral sets $\Sp(B_k-C_{k,i})$ have
the properties stated in Lemma~\ref{lem:permutations}. Then, for
such $C_{k,i}$, it is easily checked that there exists a polynomial
$A(x)$ such that $A(0)=A$ and
 $A((m+ n)k+i)=C_{k,i}$ for each $(k,i)\in {\mathscr
M}:=\{0,\dots,m\}\times \{1,\dots,n\}$.

We claim that, for each fixed $k$, $0\le k \le m$,
$\Sp(B_k-A(x))$ consists of $n$ linearly independent functions
which are analytic in a neighborhood of $x=0$.  In fact, the
spectrum of analytic perturbation $B_k-A(x)$ of $B_k$ consists
of functions which are locally analytic outside a closed
discrete set of branching points (see, e.g., \cite[Theorem
3.4.25]{aupetit}). Let $b_1,\dots,b_r$ be all the real
branching points of modulus smaller or equal to $(m+n)m+n$.
Note that for $j\in \{0,(m+n)k+i:\;(k,i)\in {\mathscr M}\}$,
$A(j)=A$ if $j=0$, and $A(j)=C_{k,i}$ otherwise, and so
$\#\Sp(B_k-A(j))=n$. Hence none of $b_i$ equals some $j\in
\{0,(m+n)k+i:\;(k,i)\in {\mathscr M}\}$. Choose a piecewise
linear path $\alpha:[0, (m+n)m+n]\to \CC$ which avoids and does
not encircle  branching points and passes through $j=(m+n)k+i$,
$(k,i)\in{\mathscr M}$. Then the spectral points
$\lambda_{k,1}(\alpha(s)),\dots,\lambda_{k,n}(\alpha(s))$ of
$\Sp(B_k-A(\alpha(s)))$ are continuous functions of $s\in
[0,n+m (m+n)]$. By the construction of $C_{k,i}$, for any
choice of $x_1\in\vec{\Sp} (B_k-C_{k,1}),\dots,x_n\in\vec{\Sp}
(B_k-C_{k,n})$, the vectors $x_1,\dots,x_n$ are linearly
independent. Then it easily follows that  the $n$ functions
$\lambda_{k,1}(\alpha(s)),\dots,\lambda_{k,n}(\alpha(s))$ are
also linearly independent for each fixed $k$. Hence,
$\lambda_{k,1}(x),\dots,\lambda_{k,n}(x)$ of $\Sp(B_k-A(x))$
are linearly independent analytic functions in a neighborhood
of the curve given by path $\alpha$. Moreover, since
$\#\Sp(A(0))=\#\Sp(A)=n$, $x=0$ is not a branching point. Thus,
also the restrictions of
$\lambda_{k,1}(x),\dots,\lambda_{k,n}(x)$ to a neighborhood of
$x=0$ are linearly independent and distinct analytic functions
since linear independence is checked by nonvanishing of
analytic function, i.e. Wronskian.

Finally, by Lemma~\ref{lem:countable},  $B_k-A(x)\in{\mathscr K}$
for every $k\in\{0,\dots,m\}$ and each $x$ outside a countable
subset of $\CC$. Hence, there exists $x$ arbitrarily close to $0$
such that $B_k-A(x)\in {\mathscr K}$ for  every $k\in\{0,\dots,m\}$.
Since $A(0)=A$, we can find $x$ such that $X=A(x)$  is close to
original matrix $A$, and that $B_k-X\in{\mathscr K}$ for each  $k$.
In particular, $B_0=0$ implies that also $X\in{\mathscr K}$. The
proof is complete.
\end{proof}

In particular, Lemma 2.3 implies that the set ${\mathscr K}$ is dense
 in $M_n(\CC)$ and hence it contains a basis of $M_n(\CC)$.

The following proposition is the main result of this section.
It gives a interesting  structural feature of basis of
$M_n(\CC)$, and is crucial for our proof of Theorem 1.1.

\begin{proposition}\label{lem:Properties-of-K2}
If $B_1, \dots, B_{n^2} \in \K$ is a basis in $M_n(\CC)$, then there
exists a basis $C_1,\dots,C_{n^2}$ in $M_n(\CC)$ such that $C_i \in
\K$ and $B_i-C_j\in{\mathscr K}$ for every $i,j \in
\{1,\dots,n^2\}$.
\end{proposition}

\begin{proof}
The matrices $C_1,\dots,C_{n^2}$ form a basis if and only if the
$n^2\times n^2$ matrix of their coefficients with respect to the
standard basis $E_{ij}$ of $M_n(\CC)$ (ordered lexicographically) is
invertible. Hence, starting with the basis $E_{ij}$ (whose matrix of
coefficients is the $n^2\times n^2$ identity matrix), the small
perturbation of $E_{ij}$ is again a basis for $M_n(\CC)$. We now use
the fact that the set
$${\mathscr K}'=\{C\in{\mathscr K}: C-B_i\in{\mathscr K}, \ i=1,2,\dots,n^2\}$$
is dense in $M_n(\CC)$ (Lemma~\ref{lem:K'}). By the density of
${\mathscr K}'$, we can
 find matrices $C_{11},C_{12}, \dots,C_{nn}\in{\mathscr K}$  with
$C_{ij}$ arbitrarily close to $E_{ij}$, such that
$C_{ij}-B_k\in{\mathscr K}$ for each $i,j,k$. Since they are
close to basis $E_{ij}$, the matrices
$C_{11},C_{12},\dots,C_{nn}$ are again a basis for $M_n(\CC)$.
\end{proof}

\section{Proof of Theorem 1.1}

In this section we  give a proof of our main result. Throughout
 we always assume that
$\Phi\colon M_n(\CC)\rightarrow M_n(\CC)$ is a map satisfying
$\Phi(0)=0$ and $\Sp(\Phi(A)-\Phi(B))\subseteq\Sp(A-B)$ for all
$A,B\in M_n(\CC)$. First  let us state four facts which were
already used in a paper by Costara~\cite{costara}.
 As usual, $\Tr(X)$ denotes
the trace of a matrix $X$.

\begin{lemma}\label{lem:step1}
 For every $X\in M_n(\CC)$, we have
$\Tr(\Phi(X))=\Tr (X).$
\end{lemma}
\begin{proof}
The proof is the same as the proof of Equation~(8) in
\cite[pp.~2675--2676]{costara}. We omit the details.
\end{proof}

\begin{lemma}\label{lem:step2}
If $A-B\in{\mathscr K}$, then $\Sp(\Phi(A)-\Phi(B))=\Sp(A-B)$,
counted with multiplicities. In particular, $A\in{\mathscr K}$
implies $\Sp\Phi(A)=\Sp A$.
\end{lemma}
\begin{proof}
The first claim follows by \cite[Lemma 5]{costara},
Lemma~\ref{lem:step1} and the linearity of trace. 
The last
claim follows by inserting $B=0$.
\end{proof}
\bigskip

 For any $X\in M_n(\CC)$, let $S_2(X)$ be the
second symmetric function in eigenvalues of a matrix $X$ (i.e.,
the coefficient of $x^{n-2}$ in characteristic polynomial
$p(x)=\det(x I-X)$).
\begin{lemma}\label{lem:step3}
For every $X\in M_n(\CC)$, we have
 $(\Tr X)^2=\Tr(X^2)+2S_2(X)$.
\end{lemma}
\begin{proof}
A straightforward calculation. Also see
\cite[Eq.(10)]{costara}.
\end{proof}

\begin{lemma}\label{lem:step4}
If $A,B, (A-B)\in{\mathscr K}$, then
$\Tr(AB)=\Tr(\Phi(A)\Phi(B))$.
 \end{lemma}
\begin{proof}
For any $A,B\in{\mathscr K}$ such that $A-B\in{\mathscr K}$,
by Lemma~\ref{lem:step3}, we have
$$\Tr\bigl((A-B)^2  \bigr)=\bigl(\Tr(A-B) \bigr)^2-2S_2(A-B).$$
Since $A-B\in{\mathscr K}$, Lemma~\ref{lem:step2} implies
$S_2(A-B)=S_2(\Phi(A)-\Phi(B))$ and
$\Tr(A-B)=\Tr(\Phi(A)-\Phi(B))$. It follows that
\begin{equation}\label{eq:Tr(AB)}
\Tr\bigl((A-B)^2  \bigr)=\bigl(\Tr(\Phi(A)-\Phi(B))^2  \bigr).
\end{equation}
Moreover, note that $A,B\in{\mathscr K}$. Likewise, $A,B\in
{\mathscr K}$ implies $\Tr (A^2)=\Tr(\Phi(A)^2)$ and $\Tr
(B^2)=\Tr(\Phi(B)^2)$. Hence, linearizing~\eqref{eq:Tr(AB)}
gives
\begin{equation}\label{eq:Tr(AB)=Tr(Phi..)}\Tr(AB)=\Tr(\Phi(A)\Phi(B))\end{equation}
whenever $A,B,A-B\in{\mathscr K}$.
\end{proof}

\begin{proof}[Proof of Theorem 1.1]%

Given a matrix $X=(x_{ij})\in M_n$, let us introduce its row vector
$R_X:=(x_{11}, x_{12},\dots,x_{1n},x_{21},\dots,x_{2n},\dots,x_{nn})$ and its column vector
$C_X:=\bigl(x_{11},x_{21},\dots,x_{n1},x_{12},\dots,x_{n2},\dots,x_{nn} \bigr)^t.$ It is elementary
that $\Tr(XY)=R_XC_Y$, and hence we may
rewrite~\eqref{eq:Tr(AB)=Tr(Phi..)} into
\begin{equation}\label{3}
R_{\Phi(A)}C_{\Phi(B)}=R_AC_B
\end{equation}
whenever $A,B,A-B\in{\mathscr K}$ (see also~Chan, Li, and Sze
\cite{chan-li-sze}).

Now, for any $A\in{\mathscr K}$, by Proposition 2.4, we can
find a basis $B_1,\dots,B_{n^2}\in{\mathscr K}$ such that
$A-B_i\in{\mathscr K}$. Using $B_i$ in place of $B$
in~\eqref{3}, we obtain a system of $n^2$ linear equations
$$R_{\Phi(A)}C_{\Phi(B_i)}=R_AC_{B_i},\quad i=1,2,\dots,n^2.$$
Introducing two $n^2\times n^2$ matrices ${\mathcal
U}:=[C_{\Phi(B_1)}|C_{\Phi(B_2)}|\dots|C_{\Phi(B_{n^2})}]$ and
${\mathcal C}:= [C_{B_1}|\dots|C_{B_{n^2}}]$, this system can be
rewritten into
\begin{equation}\label{4}
R_{\Phi(A)}{\mathcal U}=R_A{\mathcal C}.\end{equation} The
identity holds for each $A\in{\mathscr K}$ satisfying
$A-B_i\in{\mathscr K}$. By Proposition 2.4 again, there exists
another basis $A_1,\dots,A_{n^2}\in{\mathscr K}$ such that
$B_i-A_j\in{\mathscr K}$ for every $i,j\in\{1,2,\dots,n^2\}$.
Using $A_j$ in place of $A$ in~\eqref{4}, the
identity~\eqref{4} can be rewritten into a matrix equation
$${\mathcal V}{\mathcal U}={\mathcal R}{\mathcal C},$$
where ${\mathcal R}$ is an $n^2\times n^2$ matrix with $j$-th
row equal to $R_{A_j}$, and ${\mathcal V}$ is an $n^2\times
n^2$ matrix with $j$-th row  $R_{\Phi(A_j)}$. Since
$A_1,\dots,A_{n^2}$ is a basis, the matrix ${\mathcal R}$ is
invertible. Likewise, since $B_1,\dots,B_{n^2}$ is a basis,
${\mathcal C}$ is invertible. This implies invertibility of
${\mathcal U}$. In particular,~\eqref{4} yields
$$R_{\Phi(A)}=R_A{\mathcal C}{\mathcal U}^{-1}=R_A{\mathcal W};\qquad ({\mathcal W}={\mathcal C}{\mathcal U}^{-1})$$
for all $A\in {\mathscr K}$ with $A-B_i\in{\mathscr K}$.

Set~${\mathscr K}'=\{A\in{\mathscr K}: A-B_i\in{\mathscr K},\
i=1,2,\dots,n^2\}$. By Lemma~\ref{lem:K'}, the set~${\mathscr
K}'$
 is  dense in $M_n(\CC)$. Therefore we get
$$R_{\Phi(X)}=R_X{\mathcal W}, \qquad\forall \  X\in{\mathscr K}'.$$
Recall  that ${\mathcal W}={\mathcal C}{\mathcal U}^{-1}$
 is an invertible $n^2\times n^2$ matrix. Now, define a linear bijection $\Psi:M_n(\CC)\to
M_n(\CC)$ by
$$R_{\Psi(X)}=R_X{\mathcal W},  \qquad\forall \  X\in M_n(\CC).$$
The map $\Psi$ coincides with $\Phi$ on a dense
subset~${\mathscr K}'$. Moreover, Lemma~\ref{lem:step2} implies
that $\Sp(\Phi(K))=\Sp (K)$ for every $K\in{\mathscr K}'$.
Hence we have $\Sp(\Psi(K))=\Sp (K)$ for $K\in{\mathscr K}'$.
It follows from the continuity of $\Psi$ and of the spectral
function $\Sp \colon M_n(\CC)\to\CC$ that $\Psi$ is a linear bijection
satisfying  $\Sp(\Psi(X))=\Sp (X)$ for every $X\in M_n(\CC)$.
Then, by Marcus-Moyls~\cite{MM}, there exists an invertible
$S\in M_n(\CC)$ such that either $\Psi(X)=SXS^{-1}$ for all
$X\in M_n(\CC)$ or $\Psi(X)=SX^tS^{-1}$ for all $X\in
M_n(\CC)$. This gives that either
$$\Phi(K)=SKS^{-1} \quad{\rm for\ \ all} \quad K\in{\mathscr K}',$$
or $$\Phi(K)=SK^tS^{-1}\quad{\rm for\ \ all} \quad
K\in{\mathscr K}'.$$ Clearly, neither the hypothesis nor the
end result changes if we replace $\Phi$ by the map $X\mapsto
S^{-1}\Phi(X)S$ or by the map $X\mapsto (S^{-1}\Phi(X)S)^t$.
So, with no loss of generality, we can assume  that
$$\Phi(K)=K \quad{\rm for\ \ all} \quad K\in{\mathscr K}'.$$
We assert that $\Phi(X)=X$ for every $X\in M_n(\CC)$. To show this,
write $Y:=\Phi(X)$. By the assumption on $\Phi$, for any
$K\in{\mathscr K}'$, we have
$$\Sp(K-Y)=\Sp(\Phi(K)-\Phi(X))\subseteq\Sp(K-X).$$
Since ${\mathscr K}'$ is dense in $M_n(\CC)$ and as spectral
function is continuous, we derive
$$\Sp(A-Y)\subseteq\Sp(A-X)\quad{\rm for\ \ all} \quad A\in M_n(\CC).$$
For any $A\in M_n(\CC)$, let $B=A-Y$. Then the above relation yields
\begin{equation}\label{5}
\Sp(B)\subseteq \Sp(B+(Y-X))\quad{\rm for\ \ all} \quad B\in
M_n(\CC). \end{equation} We will show that $Y=X$. Assume on the
contrary that $Y-X\ne0$. In addition, for the sake of
convenience, we may also assume that $Y-X$ is already in its
Jordan form.
 If $Y-X$ is not a nilpotent matrix, then inserting
$B=-(Y-X)$ in~\eqref{5} yields a contradiction
$\{0\}\neq\Sp(Y-X)\subseteq\Sp(0)=\{0\}$. If
$Y-X=J_{n_1}\oplus\dots\oplus J_{n_k}$ is a nonzero nilpotent
matrix in its Jordan form, let
$B=(J_{n_1}^{n_1-1}\oplus\dots\oplus J_{n_k}^{n_k-1})^t$
in~\eqref{5}, where we tacitly assume that $A^0=\Id$ for any
$A\in M_n(\CC)$ including the zero matrix. Since $Y-X\ne0$,
there exists at least one block  of size $\ge 2$. Clearly then,
$B$ is not invertible, but $(Y-X)+B$ is of full rank. This
contradicts the fact that $\Sp(B)\subseteq\Sp(B+Y-X)$. Hence
$Y=X$, that is, $\Phi(X)=X$ for each $X\in M_n(\CC)$.
\end{proof}

\end{document}